\documentclass[11pt,a4paper,reqno]{amsart}
\usepackage[utf8]{inputenc}
\usepackage[T1]{fontenc}       
\usepackage{ae}   
\usepackage{amsfonts}         
\usepackage{amsmath}
\usepackage{amsthm}
\usepackage{amssymb}
\usepackage{mathtools}
\usepackage{calc}
\usepackage{centernot}
\usepackage{color}
\usepackage[pdftex]{hyperref}
\usepackage{mathrsfs}

\usepackage{bigints} 

\usepackage{mathpazo} 

\usepackage{xargs}

\usepackage{todonotes}

\newcommandx{\huom}[2][1=]{\todo[linecolor=red,backgroundcolor=red!10,bordercolor=red,#1]{#2}}

\newcommand{\ve}{\varepsilon}
\newcommand{\N}{\mathbb{N}}

\newcommand{\R}{\mathbb{R}}

\newcommand{\Ss}{\mathbb{S}}

\newcommand{\ang}[1]{\left\langle #1 \right\rangle}

\newcommand{\p}{\partial}

\DeclareMathOperator\dv{div}

\newcommand{\nb}{\bar{\nabla}}
\newcommand{\n}{\nabla}

\DeclareMathOperator\Hess{Hess}

\numberwithin{equation}{section}

\theoremstyle{plain}
\newtheorem{thm}{Theorem}[section]
\newtheorem{lem}[thm]{Lemma}
\newtheorem{cor}[thm]{Corollary}
\newtheorem{prop}[thm]{Proposition}
\newtheorem{rem}[thm]{Remark}

\theoremstyle{definition}

\newtheorem{exa}[thm]{Example}

             {\begin{list}{\arabic{enumi}.}{\usecounter{enumi}%
              \setlength{\labelsep}{0.5em}%
              \settowidth{\labelwidth}{\arabic{enumi}.}%
              \setlength{\leftmargin}{\labelwidth+\labelsep}}}%
             {\end{list}}

\author[Jean-Baptiste Casteras]{Jean-Baptiste Casteras}
\address{Jean-Baptiste Casteras,
CMAFCIO, Faculdade de Ci\^encias da Universidade de Lisboa,Edificio C6, Piso 1, Campo Grande 1749-016 Lisboa, Portugal}
\email{jeanbaptiste.casteras@gmail.com}

\author[E. Heinonen]{Esko Heinonen}
\address{E. Heinonen,
Department of Mathematics and Statistics,
P.O. Box 35, 40014 University of Jyv\"askyl\"a, Finland.
}
\email{ea.heinonen@gmail.com}

\author[I. Holopainen]{Ilkka Holopainen}
\address{I. Holopainen,
Departament of Mathematics and Statistics, P.O. Box 68 (Pietari Kalmin katu 5),
 00014 University of Helsinki, 
  Finland.
}
\email{ilkka.holopainen@helsinki.fi}

\author[J. H. De Lira]{Jorge H. De Lira}
\address{J. H. De Lira,
  Departamento de Matem\'atica,
  Universidade Federal do Cear\'a, Bloco 914, Campus do Pici,
  Fortaleza, Cear\'a, 60455-760, Brazil.
}
\email{jorge.lira@mat.ufc.br}

\subjclass[2010]{Primary 53C21, 53C44}
\keywords{Mean curvature equation, translating graphs, Hadamard manifold}

\title{Translating solitons over Cartan-Hadamard manifolds}
\date{\today}

\begin{document}

\begin{abstract}
We prove existence results for entire graphical translators of the mean curvature flow (the so-called bowl solitons) on Cartan-Hadamard manifolds. We show that the asymptotic behaviour of entire solitons depends heavily on the curvature of the manifold, and that there exist also bounded solutions if the curvature goes to minus infinity fast enough. Moreover, it is even possible to solve the asymptotic Dirichlet problem under certain conditions. 
\end{abstract}

\maketitle

\tableofcontents

\section{Introduction}
In this paper we study the existence of translating solitons in Riemannian products $N\times\R$, where $N$ is an $n$-dimensional Cartan-Hadamard manifold, i.e. a complete, simply connected Riemannian manifold with non-positive sectional curvature. A submanifold $M$ of $N\times\R$ is a translating soliton of the mean curvature flow if
\[
cX^\perp =H,
\]
where $H$ is the mean curvature vector field, $X=\p_t$ is the standard coordinate vector field of $\R$, and $c\in\R$ is a constant that indicates the velocity of the flow.
Recall from \cite[Prop. 6]{LM} that, given a domain $\Omega\subset N$ and a $C^2$ function $u\colon\Omega\to\R$, the graph 
\[
M=\big\{\big(x,u(x)\big)\colon x\in \Omega\big\}\subset N\times\R
\]
is a translating soliton if and only if $u$ satisfies the quasilinear partial differential equation
\begin{equation}\label{eq-soliton}
	\dv \frac{\nabla u}{\sqrt{1+|\nabla u|^2}} =\frac{c}{\sqrt{1+|\nabla u|^2}}
	\end{equation}
	with some constant $c\in\R$.
	
In \cite{AW} Altschuler and Wu studied surfaces over a  convex domain in $\R^2$ that are evolving by the mean curvature flow and have prescribed contact angle with the boundary cylinder. They also proved the existence of a convex rotationally symmetric translating soliton over the entire plane. Clutterbuck, Schn\"urer, and Schulze \cite{CSS} constructed entire, rotationally symmetric, strictly convex graphical translating solitons in 
$\R^{n+1},\ n\ge 2$, known as bowl solitons. 
They also classified all translating solitons of revolution giving a one-parameter family of rotationally symmetric ``winglike" solitons $M_\varepsilon$, where $\varepsilon$ represents the neck-size of the winglike soliton $M_\varepsilon$ and the limit as $\varepsilon\to 0$ consists of a dounle copy of the bowl soliton. We recall that Wang \cite{Wang-convex} characterized the bowl soliton as the only convex translating solitons in $\R^{n+1}$ that is an entire graph. Spruck and Xiao \cite{Spruck-Xiao} proved that a translating soliton which is graph over the whole $\R^2$ must be convex and hence the bowl soliton.
In recent years several families of new translating solitons in the Euclidean space have been constructed by using different techniques, see \cite{davila}, \cite{HIMW}, \cite{HMW}, \cite{nguyen1}, \cite{nguyen2}, \cite{nguyen3}. For instance, in \cite{HIMW} Hoffman, Ilmanen, Mart\'in and White gave a full classification of complete translating graphs in $\R^3$ and constructed $(n-1)$-parameter families of new examples of translating graphs in $\R^{n+1}$.

In \cite{LM} de Lira and Mart\'in extended the constructions of Clutterbuck, Schn\"urer, and Schulze to 
rotationally symmetric
Cartan-Hadamard manifolds $N_\xi$ whose metric can be written as
	\[
	dr^2 + \xi(r)^2 d\vartheta^2,
	\]
where $r=r(x) = d(o,x)$ is the distance to a pole $o\in N_\xi$ and $d\vartheta^2$ is the metric on the unit sphere $\Ss^{n-1} \subset T_o N_\xi$. They proved the existence of a one-parameter family of rotationally symmetric translating solitons $M_\varepsilon,\ \varepsilon\in [0,+\infty)$, embedded into the Riemannian product $N_\xi\times \R$. The translating soliton $M_0$, again called the bowl soliton, is the graph of an entire solution to \eqref{eq-soliton}, whereas each $M_\varepsilon,\ \varepsilon>0$, is a bi-graph over the exterior of the geodesic ball $B(o,\varepsilon)\subset N_\xi$ and is called a winglike soliton. Notice that a radial function $u=u(r)$ is a solution to \eqref{eq-soliton} if and only if it satisfies the ODE
\[
\sqrt{1+u'^2}\left(\frac{u'}{\sqrt{1+u'^2}}\right)' +u'\Delta r=c,
\]
or equivalently,
\begin{equation}\label{radeq}
\frac{u''}{1+u'^2}+(n-1)\frac{\xi'(r)}{\xi(r)}u'-c=0,
\end{equation}
where the prime $'$ denotes derivatives with respect to the radial coordinate $r$. De Lira and Mart\'in
also constructed entire grim reaper graphs on a class of complete, not necessarily rotationally symmetric, Riemannian manifolds by using Fermi coordinates attached to a geodesic.
 
Our main motivation in this paper is to prove existence results for graphical translating solitons on Cartan-Hadamard manifolds that need not be rotationally symmetric.
We assume that the radial sectional curvatures of a Cartan-Hadamard manifold $N$ satisfy    
	\begin{equation}\label{curv-as}
	-b(r(x))^2 \le K(P_x) \le - a(r(x))^2,
	\end{equation}
where $r$ is the distance function to a fixed point $o\in N$, $a,b\colon[0,\infty)\to[0,\infty)$ are smooth functions, and $P_x$ is any $2$-plane in the tangent space $T_xN$ containing $(\partial_r)_x=\nabla r(x)$. We call $a$ and $b$ the radial curvature functions of $N$.
Notice that this is not a restriction at all since there are such functions for any Cartan-Hadamard manifold.
We will denote by $N_a$ and $N_b$ the rotationally symmetric model manifolds with radial sectional curvatures $-a(r)^2$ and $-b(r)^2$, respectively. Note that the Riemannian metric $g_a$ on $N_a$ (similarly on $N_b$) can be written as
	\[
g_a = dr^2 + f_a(r)^2d\vartheta^2,
	\]
where $f_a$ is the solution for the $1$-dimensional Jacobi equation
	\begin{equation}\label{eq-jacobi}
	\begin{cases}
	f_a(0) = 0 \\
	f_a'(0) = 1 \\
	f_a'' = a^2 f_a.
	\end{cases}
	\end{equation}

Next we describe the structure of our paper and state some of our results. The main theme in Section~\ref{sec-constr} is the search of global super- and subsolutions to \eqref{eq-soliton} that stay of bounded distance from each other and use them as barriers in construction of entire solutions. First we implement the rotationally symmetric bowl solitons from the models $N_a$ and $N_b$ obtained by de Lira and Mart\'in to the actual Cartan-Hadamard manifold $N$ and obtain solutions to \eqref{eq-soliton} on geodesic balls $B(o,R)\subset N$ with constant boundary values on $\p B(o,R)$.
In Subsection~\ref{subsec-asrot} we first improve the estimate of the asymptotic behavior of a rotationally symmetric solutions obtained in \cite{CSS} and \cite{LM}.   Then we consider asymptotically rotationally symmetric manifolds and utilize the bowl solitons in rotationally symmetric models as global upper and lower barriers. If the sectional curvature upper bound goes to minus infinity fast enough, there can be bounded entire solutions to \eqref{eq-soliton}. Indeed, we have the following:
\begin{thm}\label{thm-bounded-intro}
Suppose that 
\[
K(P_x)\le -a\big(r(x)\big)^2,
\]
where the curvature upper bound goes to $-\infty$ fast enough so that
\[
\lim_{t\to\infty}\frac{a^\prime (t)}{a(t)^2}=0\ 
\text{ and }\ 
\int_0^\infty\frac{f_a(t)}{f_a^\prime (t)}dt<\infty.
\]
Then there exists a translating soliton in $N\times\R$ that is the graph of an entire bounded solution $u\colon N\to\R$ 
to the equation \eqref{eq-soliton}.
\end{thm}
Furthermore, under suitable curvature bounds, it is even possible to solve the asymptotic Dirichlet problem for \eqref{eq-soliton}. 
For instance, if the radial sectional curvatures are bounded as 
\[
- 2\cosh\big(\cosh r(x)\big)\le K(P_x)\le
- \cosh^2 r(x) -\sinh r(x)\coth\big(\sinh r(x)\big)
 \]
 and we are given a continuous function 
 $\varphi\in C(\p_\infty N)$ on the sphere at infinity, there exists a unique entire solution $u\in C^2(N)\cap C(\overline{N})$ to
 \eqref{eq-soliton} with boundary values 
 $u|\partial_\infty N = \varphi$ at infinity. See 
 Theorem~\ref{thm-adp} in Subsection~\ref{subsec-adp} for the general result. 
In Subsection~\ref{subsec-gencase} we construct global super- and subsolutions whose difference goes to zero at infinity under certain (implicit) assumptions on the Riemannian metric that are more general than those in \ref{subsec-asrot}. Then we apply these barriers in construction of entire solutions. The following two corollaries provide examples of such suitable metrics. 
\begin{cor}\label{firstcor-intro}
Let $(N^2,g)$ be a 2-dimensional Cartan-Hadamard manifold  with the Riemannian metric $g =dr^2 + h(r,\theta)^2 d\theta^2$, where 
\[
h(r,\theta)=\frac{1}{a\cos^2 \theta +b\sin^2 \theta}\sinh (ar\cos^2 \theta +b r \sin^2 \theta).
\] 
Then there exists a bowl soliton on $(N^2,g)$.
\end{cor}

\begin{cor}\label{secondcor-intro}
Let $(N^2,g)$ be a 2-dimensional Cartan-Hadamard manifold  with the Riemannian metric $g =dr^2 + h(r,\theta)^2 d\theta^2$, where 
\[
h(r,\theta)=\cos^2 \theta f_a (r) + \sin^2 \theta f_b (r),
\]
 where $f_b''(r)=\frac{\beta^2}{r^2}f_b(r)$, 
  $f_a'' (r)= \frac{\alpha^2}{r^2} f_a (r)$, with  $\beta>\alpha>4\sqrt{5}$. Then there exists a bowl soliton on $(N^2,g)$.
\end{cor}

\subsubsection*{Acknowledgements}
We would like to thank Francisco Mart\'in for several useful discussions during the preparation of the paper.

\section{Construction of entire solutions}
\label{sec-constr}
In this section we search conditions on $N$ that guarantee the existence of an entire solution to \eqref{eq-soliton}. We start with the following application of the Laplace comparison theorem and the existence result of de Lira and Mart\'in for rotationally symmetric models. 

\begin{lem}[Sub- and supersolutions]\label{sub-sup-lemma}
Assume that $N$ satisfies \eqref{curv-as}. Then there exist entire radial sub- and supersolutions to the equation \eqref{eq-soliton} for every $c\in\R$.
\end{lem}

\begin{proof}
We prove the claim for subsolutions in the case $c>0$, the other cases are similar. Let $u_a\colon N_a\to \R$ be the radial solution on $N_a$ to \eqref{eq-soliton} given by \cite[Theorem 7]{LM}.
  Using $u_a$ we define a radial function, also denoted by $u_a$, on $N$ by setting $u_a(x)=u_a(r(x))$, where $r(x)$ is the distance to the fixed point $o\in N$.
 Hence  $\nabla u_a = u_a' \nabla r$. Applying the Laplace comparison $\Delta r \ge \Delta_a r$, where $\Delta_a$ denotes the Laplace-Beltrami operator on $N_a$, and denoting $W_a=\sqrt{1+|\nabla u_a|^2}$ 
  we obtain
	\begin{align*}
	\dv \frac{\nabla u_a}{W_a} &- \frac{c}{W_a} 
	= \frac{\dv \nabla u_a}{W_a} + \ang{\nabla u_a, \nabla (W_a)^{-1}} - \frac{c}{W_a} \\
	&= \frac{u_a' \Delta r + u_a''}{W_a} + \ang{\nabla u_a, \nabla (W_a)^{-1}} - \frac{c}{W_a} \\
	&\ge \frac{u_a' \Delta_{a} r + u_a''}{W_a} + g_a\left(\nabla u_a, \nabla (W_a)^{-1}\right) - \frac{c}{W_a}\\
&=\frac{1}{W_a}\left(\frac{u_a''}{1+u_a'^2}+u_a'\Delta_a r -c\right)	
	 = 0.
	\end{align*} 
Above we used the fact $u_a'\ge 0$ that follows from the maximum principle since the radial solution $u_a$ on $N_a$ can not have interior maxima. Similarly, the entire radial solution $u_b$ on $N_b$ yields an entire supersolution on $N$. In fact, any constant function on $N$ is a supersolution for $c\ge 0$.
\end{proof}

From now on we assume, without loss of generality, that the constant $c$ in \eqref{eq-soliton} is nonnegative. 
The sub- and supersolutions $u_a$ and, respectively, $u_b$ can be uses as barriers in order to obtain solutions in geodesic balls with constant boundary values.

\begin{lem}\label{lem-ballexist}
For every geodesic  ball $B=B(o,R)\subset N$ and a constant  $m\in\R$ there exists a function 
$u\in C^2(B)\cap C(\bar{B})$ that solves the equation 
\begin{equation}\label{dprobinB}
 \begin{cases}
  \dv \dfrac{\nabla u}{\sqrt{1+|\nabla u|^2}} = \frac{c}{\sqrt{1+|\nabla u|^2}} \quad \text{in } B \\
  u|\p B = m.
 \end{cases}
\end{equation}
\end{lem}
\begin{proof}
Let $u_a$ and $u_b$ be the radial sub- and supersolutions given by Lemma \ref{sub-sup-lemma}. By adding suitable constants we may assume that $u_a=u_b=m$ on $\partial B$.
Then $u_a$ is a lower barrier and $u_b$ is an upper barrier in $B$, and therefore we obtain a priori boundary gradient estimate for the Dirichlet problem \eqref{dprobinB}. By \cite[Lemma 2.3]{CHH2} we have a priori interior gradient estimate, hence the existence of a solution to \eqref{dprobinB} follows from the 
Leray-Schauder method \cite[Theorem 13.8]{GT}.
\end{proof}

\begin{rem}
Although $u_a$ and $u_b$ are global sub- and supersolutions it seems difficult to use them as global barriers. The difficulty being that the difference 
$|u_b - u_a|$ remains bounded only in a very special case of asymptotically rotationally symmetric manifolds; see 
Lemma~\ref{lem-sub-sup-diff} and Remark~\ref{rem-assym}.
\end{rem}

\subsection{Asymptotically rotationally symmetric case}
\label{subsec-asrot}
Next we will prove the existence of entire solutions of \eqref{eq-soliton} under assumptions on the asymptotic behaviour of the radial sectional curvatures of $M$. For this we first slightly improve the estimate about the asymptotic behaviour of the rotationally symmetric solutions obtained in \cite{LM}; see also \cite{CSS}.

\begin{prop}\label{prop8}
Let $N$ be a complete rotationally symmetric Riemannian manifold whose radial sectional curvatures satisfy
	\[
	K(P_x) = -\frac{\xi''(r(x))}{\xi(r(x))} \le 0.
	\]
Suppose, furthermore, that
\begin{enumerate}
\item[(i)]
\[
\left( \frac{\xi}{\xi'} \right)'=
o\left(\min\left\lbrace\max\left\lbrace 1,
\left(\frac{\xi}{\xi'}\right)^2\right\rbrace,\max\left\lbrace\frac{\xi}{\xi'},\frac{\xi'}{\xi}\right\rbrace\right\rbrace\right)
\]
\item[(ii)]
\[
 h\left( \frac{\xi}{\xi'} \right)'=o\left(\max\left\lbrace\frac{\xi}{\xi'},\frac{\xi'}{\xi}\right\rbrace\right) 
\]
\item[(iii)]
\[
\frac{h'}{h}=o\left(\max\left\lbrace\frac{\xi'}{\xi},\frac{\xi}{\xi'}\right\rbrace\right) 
\]
\end{enumerate}
	as $r\to \infty$ for some smooth positive $h$.
	Then the rotationally symmetric translating solitons $M_\varepsilon$, $\varepsilon\in [0,+\infty)$, are described, outside a cylinder over a geodesic ball $B_R(o)\subset N$, as graphs or bi-graphs of functions with the following asymptotic behavior
	\begin{equation}\label{comp-behav}
	u'(r) = \frac{c}{n-1} \frac{\xi(r)}{\xi'(r)} + 
	o\left(\frac{1}{h(r)}\right)
	\end{equation}
	as $r\to +\infty$.
\end{prop}

\begin{proof}
Denoting $\varphi=u'$ the equation \eqref{radeq} becomes 
\[
\varphi'(r)=\big(1+\varphi^2(r)\big)\left(c-(n-1)
\frac{\xi'(r)}{\xi(r)}\varphi\right) =:F\big(r,\varphi(r)\big).
\]
For $\epsilon>0$ we denote
\[
\zeta(r)=(1-\epsilon)\frac{c}{n-1}\frac{\xi(r)}{\xi'(r)}.
\]
We claim that for every given $\epsilon>0$ and $r_0>R$ there exists $r_1>r_0$ such that
\[
\zeta(r_1)<\varphi(r_1).
\]
Indeed, if this were not the case, there would exist $\epsilon>0$ and $r_0>R$ such that 
\[
\varphi(r)\le (1-\epsilon)\frac{c}{n-1}\frac{\xi(r)}{\xi'(r)}
\]
for every $r>r_0$. In this case we would have
\[
\varphi'(r)\ge c\epsilon \big(1+\varphi^2(r)\big),\quad
r>r_0
\]
which implies that
\[
c\epsilon(r-r_\ast)\le \arctan\varphi(r)-\arctan\varphi(r_\ast)
\]
for all $r>r_\ast>r_0$.
This contradicts the fact that the solution is entire and hence 
$r\to+\infty$.

Next we claim that 
\[
\zeta'(r)< F\big(r,\zeta(r)\big)
\] 
for all sufficiently large $r>R$. Indeed, we have
\[
F\big(r,\zeta(r)\big)=c\epsilon\left(1+(1-\epsilon)^2\frac{c^2}{(n-1)^2}\left(\frac{\xi(r)}{\xi'(r)}\right)^2\right)
> (1-\epsilon)\frac{c}{n-1}\left(\frac{\xi(r)}{\xi'(r)}\right)'
\]
if and only if 
\[
(n-1)^2 +(1-\epsilon)^2c^2\left(\frac{\xi(r)}{\xi'(r)}\right)^2
>
\left(\frac{1}{\epsilon}-1\right)(n-1)\left(\frac{\xi(r)}{\xi'(r)}\right)'
\]
that holds true for all sufficiently large $r>R$ 
by the assumption (i).  
Then adjusting $r_1>r_0$ to be sufficiently large we conclude from a standard comparison argument for nonlinear ODEs that
\[
\zeta(r)<\varphi(r)
\]
for every $r>r_1$, with $r_1>r_0>R$ sufficiently large. We conclude that for every given $\epsilon>0$ and $r_0>R$ there exists $r_1>r_0$ such that
\[
(1-\epsilon)\frac{c}{n-1}\frac{\xi(r)}{\xi'(r)}<\varphi(r)
\]
for all $r>r_1$.

Similarly, given $\epsilon>0$ we denote
\[
\eta(r)=(1+\epsilon)\frac{c}{n-1}\frac{\xi(r)}{\xi'(r)}
\]
Again we claim that for every $\epsilon>0$ and $r_0>R$ there exists $r_1>r_0$ such that 
\[
\eta(r_1)>\varphi(r_1).
\]
Otherwise, we could find $\epsilon>0$ and $r_0>R$ such that
\[
\varphi(r)>(1+\epsilon)\frac{c}{n-1}\frac{\xi(r)}{\xi'(r)}
\]
for every $r>r_0$. Consequently,
\[
\varphi'(r)<-c\epsilon\big(1+\varphi^2(r)\big),\quad
r>r_\ast>r_0,
\]
leading to a contradiction
\[
\arctan\varphi(r)-\arctan\varphi(r_\ast)<-c\epsilon(r-r_\ast),
\quad r<r_\ast<r_0.
\]
Next we prove that 
\begin{align*}
\eta'(r)&=(1+\epsilon)\frac{c}{n-1}\left(\frac{\xi(r)}{\xi'(r)}\right)'\\
&>
-c\epsilon\left(1+(1+\epsilon)^2\frac{c^2}{(n-1)^2}\left(\frac{\xi(r)}{\xi'(r)}\right)^2\right)\\
&= F\big(r,\eta(r)\big)
\end{align*}
for all sufficiently large $r$, or equivalently that
\[
-(n-1)^2-(1+\epsilon)^2 c^2\left(\frac{\xi(r)}{\xi'(r)}\right)^2
<
\left(\frac{1}{\epsilon}+1\right)(n-1)\left(\frac{\xi(r)}{\xi'(r)}\right)'
\]
that holds true for all sufficiently large $r$ 
again by the assumption (i). 
Therefore we can again 
conclude that for every given $\epsilon>0$ and $r_0>R$ there exists $r_1>r_0$ such that
\[
\varphi(r)<(1+\epsilon)\frac{c}{n-1}\frac{\xi(r)}{\xi'(r)}
\]
for all $r>r_1$.

We set
\[
\varphi(r)=\frac{c}{n-1}\frac{\xi(r)}{\xi'(r)}+\psi(r)
\]
and claim that
\[
\lim_{r\to+\infty}\psi(r)=0.
\]
If this were not the case, then 
\[
\liminf_{r\to+\infty}\psi(r)<0
\]
or 
\[
\limsup_{r\to+\infty}\psi(r)>0.
\]
Note that 
\[
\psi'(r)=-(n-1)\psi(r)\frac{\xi'(r)}{\xi(r)}\left(1+
\left(\frac{c}{n-1}\frac{\xi(r)}{\xi'(r)}+\psi(r)\right)^2\right)
-\frac{c}{n-1}\left(\frac{\xi(r)}{\xi'(r)}\right)'
\]
and that
\[
|\psi(r)|\le\frac{c\epsilon}{n-1}\frac{\xi(r)}{\xi'(r)}
\]
for $r>r_1$.

Suppose first that 
\[
\liminf_{r\to+\infty}\psi(r)<0.
\]
Given any $\delta>0$, with
\[
\liminf_{r\to+\infty}\psi(r)<-\delta,
\]
there are arbitrary large $r_\ast$ 
such that 
\[
\psi(r_\ast) < -\delta.
\]
For every such $r_\ast\ge r_2$ we have, by the assumption (i), that
\begin{align*}
\psi'(r_\ast)& > 
\delta\left(\frac{(n-1)\xi'(r)}{\xi(r)} 
+ \frac{(1-\epsilon)^2 c^2}{n-1}\frac{\xi(r)}{\xi'(r)}\right)-\frac{c}{n-1}\left( \frac{\xi(r)}{\xi'(r)} \right)'\\
&\ge \tilde{\delta}>0 
\end{align*}  
whenever $r_2$ is large enough. 
Hence there exists an interval
$[r_\ast, \tilde{r}],\ \tilde{r}>r_\ast$, such that 
$\psi(r)\le -\delta,\ \psi'(r)\ge \tilde{\delta}$, and hence 
$\psi(r)\ge (r-r_\ast)\tilde{\delta} +\psi(r_\ast)$ for all 
$r\in [r_\ast,\tilde{r}]$. Let $r_3$ be the supremum of such 
$\tilde{r}$. Then $\psi(r_3)=-\delta,\ \psi'(r_3)\ge \tilde{\delta}$, and consequently $\psi>-\delta$ on some open interval 
$(r_3,r)$. Since 
 \[
\liminf_{r\to+\infty}\psi(r)<-\delta,
\]
the supremum
\[
r_4:=\sup\{r>r_3\colon \psi(t)>-\delta\ \forall t\in (r_3,r)\}
\]
is finite, hence $\psi(r_4)=-\delta$ and $\psi'(r_4)\ge\tilde{\delta}$. This leads to a contradiction since 
$\psi(t)>-\delta$ for all $t\in (r_3,r_4)$.

Suppose then that
\[
\limsup_{r\to+\infty}\psi(r)>0. 
\]
Given any $\delta>0$, with
\[
\delta<\limsup_{r\to+\infty}\psi(r),
\]
there are arbitrary large $r_\ast$ 
such that 
\[
\psi(r_\ast) > \delta.
\]
Again, for every such $r_\ast\ge r_2$ we have
\[
\psi'(r_\ast)<-\tilde{\delta}
\]
whenever $r_2$ is large enough. As above, this leads to a contradiction. 

Thus
\[
\lim_{r\to+\infty}\psi(r)=0.
\]

Next we study the rate of the convergence $\psi \to 0$.
For this, we denote
	\begin{equation}\label{lambda-def}
	\lambda(r) = h(r) \psi(r)
	\end{equation}
and show that $\lambda(r)\to 0$ as $r\to \infty$. 
Since
\[
\lambda'(r)=h(r)\psi'(r)
+ h'(r)\psi(r)
\]
and
	\[
	\psi'(r) = -(n-1) \psi(r) \frac{\xi'(r)}{\xi(r)} \left( 1 + \Big(\frac{c}{n-1} \frac{\xi(r)}{\xi'(r)} + \psi(r) \Big)^2 \right) - \frac{c}{n-1} \left( \frac{\xi(r)}{\xi'(r)} \right)'
	\]
we have
	\begin{align*}
	\lambda'(r)
	&= - (n-1)\lambda(r) \frac{\xi'(r)}{\xi(r)} \left( 1 + \Big(\frac{c}{n-1} \frac{\xi(r)}{\xi'(r)} + \psi(r) \Big)^2 \right) \\
		&\quad - \frac{c\cdot h(r)}{n-1}\left( \frac{\xi(r)}{\xi'(r)} \right)' + 
		\frac{h'(r)\lambda(r)}{h(r)}\\
&=-\lambda(r)\left(\frac{(n-1)\xi'(r)}{\xi(r)} - \frac{h'(r)}{h(r)} + \frac{(n-1)\xi'(r)}{\xi(r)}
\left(\frac{c}{n-1}\frac{\xi(r)}{\xi'(r)} +\psi(r)\right)^2\right)\\
&\quad - \frac{c\cdot h(r)}{n-1}\left( \frac{\xi(r)}{\xi'(r)} \right)' .		
	\end{align*}
Assuming that $\lambda(r) \le -\delta <0$ for some arbitrarily large $r$ implies that 
\begin{align*}
\lambda'(r) &\ge 
\delta\left(\frac{(n-1)\xi'(r)}{\xi(r)} 
- \frac{h'(r)}{h(r)} + \frac{(1-\epsilon)^2 c^2}{n-1}\frac{\xi(r)}{\xi'(r)}\right)-\frac{c\cdot h(r)}{n-1}\left( \frac{\xi(r)}{\xi'(r)} \right)'\\
&\ge \tilde{\delta}>0 
\end{align*}
for such $r$ by assumptions (ii) and (iii).
By a similar argument as before this implies that 
\[
\liminf_{r\to +\infty}\lambda(r)\ge 0.
\]
On the other hand, if $\lambda(r)\ge\delta>0$ for some arbitrarily large $r$, then $\lambda'(r)\le-\tilde{\delta}<0$
for such $r$, and again we can conclude that 
\[
\limsup_{r\to +\infty}\lambda(r)\le 0.
\]
We have proven that $\lambda(r)\to 0$ as $r\to+\infty$, and therefore
\[
\varphi(r)=\frac{c}{n-1}\frac{\xi(r)}{\xi'(r)}+
o\left(\frac{1}{h(r)}\right)	
\]
as $r\to+\infty$.
\end{proof}

\begin{exa}
\begin{enumerate}
\item[(a)]
If $\xi(r)=r$, we may choose $h(r)=r/\log r$.
\item[(b)]
If $\xi(r)=\sinh r$, we may choose, for example, 
\[
h(r)=e^{r^\alpha},\quad 0<\alpha<1.
\]
In fact, also $h(r)=\sinh r$ will do as can be seen in the estimation of $\lambda'(r)$.
\item[(c)]
 If $\xi(r)=\sinh(\sinh r)$, we can choose 
 $h(r)=e^{r}/\log r$. 
\end{enumerate}
\end{exa}

Next, let us discuss the condition (i) of Proposition \ref{prop8}. Before doing so, we need the following lemma  which provides criteria to compare the sectional curvature $K(P_x)$ with $(\xi^\prime/\xi)\big(r(x)\big)$ when $r(x)$ is very large. To simplify notation, we set $-a^2=K$ and $f_a=\xi$.

\begin{lem}\label{lem2.6}
Suppose that $a: [0,\infty ) \rightarrow [0,\infty )$ is a smooth function such that $a(t)>0$ for every $t\geq t_0$, for some $t_0>0$.
\begin{enumerate}
\item[(i)] If
 \begin{equation}
\label{G1}
\lim_{t\rightarrow \infty} \dfrac{a' (t)}{a^2 (t)}=0,
\end{equation}
 then $$\lim_{t\rightarrow \infty}\dfrac{f_a^\prime (t)}{a(t)f_a (t)}=1.$$
\item[(ii)] If \eqref{G1} does not hold, we assume that
$ a^\prime (t) \leq -c\,a^2(t)< 0,$ with some constant $c>0$, for all sufficiently large every $t$ and, furthermore, either
\begin{enumerate}
\item[(a)] $a\not\in L^1\big([0,\infty)\big)$ or
\item[(b)] $a\in L^1\big([0,\infty)\big)$ and $f_a^\prime(t)\to \infty$ as $t\to\infty$.
\end{enumerate}
 Then we have
\begin{equation}\label{propcomp}
\liminf_{t\rightarrow \infty}\dfrac{f_a^\prime (t)}{a(t)f_a (t)}>1.
\end{equation}
\end{enumerate}
\end{lem}
\begin{proof}
The proof of (i) can be found in \cite[Lemma 2.3]{HoVa}. Let us prove next (ii). 
Its proof follows closely \cite[Lemma 2.3]{HoVa}.
For $k>0$ we define  
\[
g_k(t)=\exp\left(k\int_0^t a(s) ds\right)
\]
and notice that
\[
g_k^\prime (t)=k\,a(t)g_k(t)
\]
and
\[
g_k''(t)=\big(k\,a^\prime(t)+k^2a^2(t)\big)g_k(t).
\]
Since $a^\prime (t) \leq -c\,a^2(t)< 0$ for 
all $t\ge t_1$, we obtain
\begin{align*}
\frac{g_k''(t)}{g_k(t)}&=k\,a^\prime (t)+k^2a^2(t)\\
&\le \big(k^2 -c\, k\big)a^2(t)\\
&=a^2(t)
\end{align*}
if 
\[
k=\sqrt{1+(c/2)^2}+c/2=: 1+\varepsilon >1.
\]
Therefore
\begin{align*}
\big(f_a g_{1+\varepsilon}^\prime - 
g_{1+\varepsilon}f_a^\prime\big)^\prime(t) &=
f_a(t)g_{1+\varepsilon}''(t)-g_{1+\varepsilon}(t)f_a''(t)\\
& =f_a(t)\big(g_{1+\varepsilon}''(t)-g_{1+\varepsilon}(t)a^2(t)\big)\le 0
\end{align*}
for all $t\ge t_1$. Hence
\[
f_a(t)g_{1+\varepsilon}^\prime (t)\le C+
f_a^\prime (t)g_{1+\varepsilon}(t),
\]
and consequently
\[
\dfrac{g_{1+\varepsilon}^\prime (t)/g_{1+\varepsilon}(t)}{f_a^\prime (t)/f_a(t)}\le 1+\frac{C}{f_a^\prime (t) g_{1+\varepsilon}(t)} 
\]
for all $t\ge t_1$. 
If $a\not\in L^1\big([0,\infty)\big)$, we have $g_{1+\varepsilon}(t)\to\infty$ as $t\to\infty$. Otherwise
$f_a^\prime (t)\to\infty$ as $t\to\infty$ by the assumption (b) of (ii). Thus in both cases we have
\[
\limsup_{t\to\infty}\frac{a(t)}{f_a^\prime (t)/f_a(t)}
=\frac{1}{1+\varepsilon}\limsup_{t\to\infty}
\dfrac{g_{1+\varepsilon}^\prime (t)/g_{1+\varepsilon}(t)}{f_a^\prime (t)/f_a(t)}\le \frac{1}{1+\varepsilon}.
\]
This proves \eqref{propcomp}. 
\end{proof}

Notice that 
\begin{equation}
\label{dere1}
\left(\frac{f_a}{f_a^\prime}\right)^\prime = 1 - \left(\frac{f_a}{f_a^\prime}\right)^2 a^2.
\end{equation}
 So if 
\[
\lim_{t\rightarrow \infty} \frac{a' (t)}{a^2 (t)}=0,
\] 
then we deduce from Lemma \ref{lem2.6} that
$$\left(\frac{f_a}{f_a^\prime}\right)^\prime (t)=o(1),$$
as $t\to\infty$. Furthermore, using once more Lemma \ref{lem2.6}, we get 
\[ 
\lim_{t\rightarrow \infty}\frac{f_a^\prime}{f_a}(t)=\infty
\]
if $\lim_{t\rightarrow \infty}a(t) =+\infty$, 
and 
\[
\lim_{t\rightarrow \infty}\frac{f_a}{f_a^\prime}(t)=\infty
\]
if $\lim_{t\rightarrow \infty}a(t) =0$.
Therefore, if  
\[
\lim_{t\rightarrow \infty} \frac{a' (t)}{a^2 (t)}=0,
\]
 then the condition (i) of Proposition \ref{prop8} is satisfied.

Next, let us assume that $a$ satisfies the assumptions (ii) in Lemma~\ref{lem2.6}. First, integrating the inequality $a^\prime (t)\leq - c a^2 (t)$, we deduce that there exists a constant $C$ such that $a(t) \leq C/t=\tilde{a}(t)$. Therefore, using Example 2.1 of \cite{HoVa}, there exist two positive constants $\alpha>1$ and $c$ such that $\lim_{t\rightarrow \infty} \dfrac{f_{\tilde{a}} (t)}{t^\alpha}=c$ and $\lim_{t\rightarrow \infty} \dfrac{f^\prime_{\tilde{a}} (t)}{t^{\alpha-1}}=\alpha c$. This implies that
 $$\lim_{t\to \infty}\frac{f_{\tilde{a}}}{f_{\tilde{a}}^\prime} (t) =\infty.$$
By standard comparison theorem (see for instance Lemma 2.2 of \cite{HoVa}), we deduce that 
 \begin{equation}
\label{stcompe1}
\lim_{t\to \infty}\frac{f_{a}}{f_{a}^\prime} (t) =\infty.
\end{equation}
Since, by Lemma \ref{lem2.6} and \eqref{dere1}, $\left(\frac{f_a}{f_a^\prime}\right)^\prime (t) \geq C>0 $, we conclude that $(i)$ of Proposition \ref{prop8} holds if $a$ satisfies the assumptions (ii) in Lemma~\ref{lem2.6}.

Concerning conditions (ii) and (iii) of Proposition \ref{prop8}, we see that if \\
$\lim_{t\to \infty}a(t)=\infty$ then we can take $h= 
f_a^\prime /f_a$. Indeed, (ii) is a direct consequence of the fact that $(f_a /f_a^\prime )^\prime (t)\rightarrow 0$ as $t\rightarrow \infty$. Concerning (iii), we notice that 
\[
h^\prime = a^2 - \left(\frac{f_a^\prime}{f_a}\right)^2 = 
\left(\frac{f_a^\prime}{f_a}\right)^2  
\left( \left(\frac{af_a}{f_a^\prime}\right)^2 -1 \right). 
\] 
Using that 
\[
\lim_{t\rightarrow \infty}\frac{f_a^\prime (t)}{a(t)f_a (t)}=1,
\]
 we deduce that $h^\prime =o(h^2)$ which is equivalent to (iii). We do not comment on these conditions in the case $\lim_{t\to \infty}a(t)=0$ since in any case, we will need a refined asymptotic expansion.

\begin{prop}\label{prop8x}
Let $N$ be a complete rotationally symmetric Riemannian manifold whose radial sectional curvatures satisfy
	\[
	K(P_x) = -\frac{\xi''(r(x))}{\xi(r(x))} \le 0.
	\]
Suppose, furthermore, that
	\[
	\left( \frac{\xi}{\xi'} \right)'(r)
	= O(1), \quad \left( \frac{\xi}{\xi'} \right)''(r)
	=o\left( \left( \frac{\xi}{\xi'} \right)'(r) \left( \frac{\xi}{\xi'} \right)(r) \right),
	\]
and
\[
\frac{\xi(r)}{\xi'(r)}\to \infty
\]
as $r\to\infty$.
	Then the rotationally symmetric translating solitons $M_\varepsilon$, $\varepsilon\in [0,+\infty)$, are described, outside a cylinder over a geodesic ball $B_R(o)\subset N$, as graphs or bi-graphs of functions with the following asymptotic behavior
	\begin{equation}\label{comp-behav2}
	u'(r) = \frac{c}{n-1} \frac{\xi(r)}{\xi'(r)}
+\lambda(r)\frac{\xi'(r)}{\xi(r)}+ \eta(r)\left(\frac{\xi'(r)}{\xi(r)}\right)^3
\end{equation}	
		as $r\to +\infty$, where
		\begin{equation}\label{limlam}
\left|\lambda+\frac{1}{c}\left(\frac{\xi}{\xi'}\right)'\right|\to 0
\end{equation}
and
\begin{equation}\label{limeta}
\left| \eta +\frac{3(n-1)}{c^3}\left(\left(\frac{\xi}{\xi'}\right)'\right)^2 - \frac{(n-1)^2}{c^3}\left(\frac{\xi}{\xi'}\right)'-\frac{n-1}{c^3}\left(\frac{\xi}{\xi'}\right)\left(\frac{\xi}{\xi'}\right)''\right| \to 0
\end{equation}
as $r\to+\infty$.
\end{prop}
\begin{proof}
By the proof of \ref{prop8}
\[
u'(r)=\varphi(r)= \frac{c}{n-1} \frac{\xi(r)}{\xi'(r)}
+\psi(r),
\]
where $\psi(r)\to 0$ as $r\to\infty$. We write
\[
\lambda(r)=\frac{\xi(r)}{\xi'(r)}\psi(r)
\]
and claim that \eqref{limlam} holds for $\lambda$.
First we compute
\begin{align*}
\lambda'&=\frac{\xi}{\xi'}\psi' + 
\left(\frac{\xi}{\xi'}\right)'\psi\\
&=-(n-1)\psi\left(1+
\left(\frac{c}{n-1}\frac{\xi}{\xi'}+\psi\right)^2\right)
-\frac{c}{n-1}\frac{\xi}{\xi'}\left(\frac{\xi}{\xi'}\right)' +\left(\frac{\xi}{\xi'}\right)'\psi\\
&=\frac{c}{n-1}\frac{\xi}{\xi'}\left(-\left(\frac{\xi}{\xi'}\right)' - c\lambda - 2(n-1)\psi^2 
+\frac{n-1}{c}\frac{\xi'}{\xi}\left(\frac{\xi}{\xi'}\right)'\psi \right)\\
&\qquad -(n-1)\psi\big(1+\psi^2\big).
\end{align*}
If, for fixed $\delta>0$
\[
\lambda(r)>-\frac{1}{c}\left(\frac{\xi}{\xi'}\right)'(r)+\delta
\]
for some arbitrary large $r$, then $\lambda'(r)<- \tilde{\delta}<0$ for such $r$. Similarly, if 
\[
\lambda(r)<-\frac{1}{c}\left(\frac{\xi}{\xi'}\right)'(r)-\delta
\]
for some arbitrary large $r$, then 
$\lambda'(r)>\tilde{\delta}$ for such $r$. Hence \eqref{limlam} holds.

Next we write 
\[
\lambda=-\frac{1}{c}\left(\frac{\xi}{\xi'}\right)'
+\eta\left(\frac{\xi'}{\xi}\right)^2.
\]
Hence 
\begin{align*}
\eta'&=\frac{c^2}{n-1}\frac{\xi}{\xi'}\biggl\lbrace -\eta + \frac{2(n-1)}{c^2}\left(\lambda+\frac{1}{c}\left(\frac{\xi}{\xi'}\right)'\right)\left(\frac{\xi}{\xi'}\right)' - \frac{2(n-1)}{c}\lambda^2 \\
&\quad +\frac{n-1}{c^2}\lambda\left(\frac{\xi}{\xi'}\right)' 
 -\frac{(n-1)^2}{c^2}\lambda(1+\psi^2) + \frac{n-1}{c^3}\left(\frac{\xi}{\xi'}\right)\left(\frac{\xi}{\xi'}\right)''\biggr\rbrace
\end{align*}
from which we deduce as earlier that \eqref{limeta} holds.
\end{proof}

We claim that the assumptions of the previous lemma hold under the hypothesis of Lemma \ref{lem2.6} (ii) with $K=-a^2$. By Lemma \ref{lem2.6}, \eqref{dere1} and \eqref{stcompe1}, we see that 
\[
 \left( \frac{f_a}{f_a^\prime} \right)'(r) \geq C
 \] 
 and  
\[
 \frac{f_a}{f_a'}(r) \rightarrow \infty
 \] 
 as $r\to \infty$. Notice that
$$\left( \frac{f_a}{f_a'} \right)''=-2 \dfrac{f_a}{f_a '} a \left(\left(\dfrac{f_a}{f_a '}\right)' a + \dfrac{f_a}{f_a^\prime}a^\prime\right).$$
So the condition 
$$\left( \frac{f_a}{f_a'} \right)''(r)=o\left( \left( \frac{f_a}{f_a'} \right)'(r) \left( \frac{f_a}{f_a'} \right)(r) \right)$$
is equivalent to
$$\dfrac{f_a}{f_a^\prime} a a^\prime =o(1).$$
Since $ af_a/f_a^\prime \leq 1$ and $a^\prime$ goes to $0$, this holds true. This proves the claim.

\begin{lem}\label{lem-sub-sup-diff}
Let $N$ be a Cartan-Hadamard manifold whose radial sectional curvatures satisfy
	\[
	-b(r)^2 \le K_N \le -a(r)^2 \le -\alpha^2,
	\]
where $\alpha>0$ is constant and $b\ge a$ are positive smooth functions such that the corresponding solutions to the Jacobi equation \eqref{eq-jacobi} satisfy
	\begin{equation}\label{asymptotic-as}
	\frac{f_a}{f_a'}(r) = \frac{f_b}{f_b'}(r) + 
	O\left(\frac{1}{h(r)}\right)
	\end{equation}
	and that
	\begin{equation}\label{extra-curv-as}
	h(r) \left( \frac{f_i}{f_i'}\right)'(r)\to 0 , \quad i=a,b,
	\end{equation}
for some smooth, increasing, positive function $h$ such that 
\begin{equation}\label{intva}
\frac{1}{h}\in L^1 (+\infty)
\end{equation} 
and
\begin{equation}\label{h'h}
\frac{h'}{h}(r)\to 0
\end{equation}
as $r\to\infty$.	
Then there exist entire radial sub- and supersolutions, $u_a$ and $u_b$, of \eqref{eq-soliton} such that the difference $|u_b(x) - u_a(x)|$ remains bounded as $r(x) \to \infty$.
\end{lem}

\begin{proof}
Since 
\[
\frac{f'_b}{f_b}(s)\ge \frac{f'_a}{f_a}(s)\ge\frac{1}{\alpha}\coth(\alpha s)
\]
by \cite[Lemma 2.2]{HoVa}, we see that the assumptions in Proposition~\ref{prop8} hold with a function $h$ satisfying \eqref{extra-curv-as} and \eqref{h'h}.
Let $u_a$ and $u_b$ be the entire radial sub- and, respectively, supersolution on $N$ obtained from the model manifolds $N_a$ and $N_b$ as in 
Lemma~\ref{sub-sup-lemma}. 
By considering $u_i-u_i(o),\ i=a,b$, we can assume that $u_i(o)=0$ and hence, using \eqref{comp-behav},  \eqref{asymptotic-as} and \eqref{intva},  the difference of the these functions can be estimated by
	\begin{align*}
	|u_b(x) - u_a(x)| 
&= \left| \int_0^{r(x)}\left( u_b'(s) - u_a'(s)\right) \, ds \right| \\
&= \frac{c}{n-1}  \left| \int_0^{r(x)}
\left( \frac{f_b}{f_b'}(s) - \frac{f_a}{f_a'}(s) +
O\left(\frac{1}{h(s)}\right)\right) \, ds \right| \\
	&\le C < \infty.
	\end{align*}
 \end{proof}
\begin{rem}\label{rem-assym}
The assumptions in Lemma~\ref{lem-sub-sup-diff} are very restrictive.   Indeed, it follows from 
\eqref{asymptotic-as}, \eqref{extra-curv-as}, and \eqref{intva} that
\[
 b^2 - a^2 
 = \left(\frac{f'_b}{f_b}\right)^2\left(1-\left(\frac{f_b}{f'_b}\right)'\right) -
\left(\frac{f'_a}{f_a}\right)^2
\left(1-\left(\frac{f_a}{f'_a}\right)'\right)
\to 0
\]
as $r\to\infty$.
\end{rem}
\begin{thm}\label{thm-globexist}
Let $N$ be a Cartan-Hadamard manifold as in 
Lemma~\ref{lem-sub-sup-diff}.
 Then there exists an entire solution $u \colon M \to \R$ to the soliton equation \eqref{eq-soliton}. 
\end{thm}

\begin{proof}
Let $u_a$ and $u_b$ be the entire radial sub- and, respectively, supersolution on $N$ obtained from the model manifolds $N_a$ and $N_b$ as in 
Lemma~\ref{sub-sup-lemma}. By Lemma~\ref{lem-ballexist} there exists, for each $k\in\N$, a function 
$u_k\in C^2\big(B(o,k)\big)\cap C\big(\bar{B}(o,k)\big)$ that solves the equation 
\[
 \begin{cases}
  \dv \dfrac{\nabla v}{\sqrt{1+|\nabla v|^2}} = \frac{c}{\sqrt{1+|\nabla v|^2}} \quad \text{in } B(o,k) \\
  v|\p B(o,k) = m_k,
 \end{cases}
\]
where $m_k=u_a|\p B(o,k)$. On the other hand, by Lemma~\ref{lem-sub-sup-diff}, there exists a constant $C>0$ such that $u_b+C\ge u_a$ on $N$. Furthermore, $u_b +C$ is a (global) supersolution. Therefore the sequence $(u_k)$ is locally uniformly bounded and hence there exists a subsequence converging locally uniformly with respect to $C^2$-norm to an entire solution $u$.
\end{proof}

As in \cite{LM} we call the graph $M=\{(x,u(x)) \in N\times\R \colon x\in N \}$ a bowl soliton.

We want to give some examples about metrics that satisfy the assumptions \eqref{asymptotic-as} -- \eqref{h'h}.

\begin{exa}
\begin{enumerate}
\item[1.] We may choose $f_a(r)=\sinh r$ and $f_b(r)=g(r)\sinh r$, where $g$ is a smooth positive function such that $g(r)=c\, e^{-1/r}$ for all large $r$ with a suitable positive constant $c$ and that 
\[
\frac{2g^\prime(r)\coth r}{g(r)}+\frac{g''(r)}{g(r)}
\]
is nonnegative and bounded and, furthermore, that $g(0)=1$ and $g^{(\text{odd})}(0)=0$.
Then 
\[
a^2(r)=\frac{f_a''(r)}{f_a(r)}=1\le 1+
\frac{2g^\prime(r)\coth r}{g(r)}+\frac{g''(r)}{g(r)}
=\frac{f_b''(r)}{f_b(r)}=b^2(r),
\]
and $f_a$ and $f_b$ satisfy the assumptions \eqref{asymptotic-as}--\eqref{h'h} with $h(r)=r^2$.
\item[2.]
Another example is given by $f_a(r)=\tfrac{1}{2}\big(\sinh r +\tfrac{1}{2}\sinh 2r\big)$ and \newline
$f_b(r)=\tfrac{1}{2}\sinh 2r$. Then
\[
a^2(r)=\frac{f_a''(r)}{f_a(r)}=\frac{\sinh r +2\sinh(2r)}{\sinh r +\tfrac{1}{2}\sinh 2r}\le 4=\frac{f_b''(r)}{f_b(r)}=b^2(r)
\]
and, furthermore, $f_a$ and $f_b$ satisfy \eqref{asymptotic-as}--\eqref{h'h} with $h(r)=e^r$.\\
In fact, any choice
\[
f_a(r)=s\,\sinh r+(1-s)\tfrac 1b \sinh(bt)\text{ and }
f_b(t)=\tfrac 1b \sinh bt,
\]
with $0<s<1$ and $b>1$, will do.
\end{enumerate}
\end{exa}


\subsection{Bounded solutions}
\label{subsec-bddsol}

Next we show that if the sectional curvatures of $N$ are negative enough near the infinity, it is possible to have entire bounded solutions of \eqref{eq-soliton}. Recall that, in the radially symmetric case, the soliton equation 
\eqref{eq-soliton} can be written as
	\begin{equation}\label{rot-sym-sol}
	\frac{u''(r)}{1+u'(r)^2} + (n-1)\frac{\xi'(r)}{\xi(r)} u'(r) - c = 0.
	\end{equation}

\begin{thm}\label{thm-bounded}
Suppose that 
\[
K(P_x)\le -a\big(r(x)\big)^2,
\]
where the curvature upper bound goes to $-\infty$ fast enough so that
\[
\lim_{t\to\infty}\frac{a^\prime (t)}{a(t)^2}=0\ 
\text{ and }\ 
\int_0^\infty\frac{f_a(t)}{f_a^\prime (t)}dt<\infty.
\]
Then there exists a translating soliton in $N\times\R$ that is the graph of an entire bounded solution $u\colon N\to\R$ 
to the equation \eqref{eq-soliton}.
\end{thm}
\begin{rem}
For instance the rotationally symmetric manifold $(M,g)$, where
\[
g= dr^2 +\sinh^2(\sinh r)d\vartheta^2,
\]
satisfies the assumptions above.
\end{rem}
\begin{proof}
We will prove that the bounded function
\[
v(x)=c_1\int_0^{r(x)}\frac{f_a(t)}{(n-1)f_a^\prime(t)}dt
\]
is a subsolution to \eqref{eq-soliton} if the constant $c_1$ is large enough. It follows from the Laplace comparison 
\[
\Delta r(x)\ge \frac{(n-1)f_a^\prime\big(r(x)\big)}{f_a\big(r(x)\big)}
\]
that $v$ is a subsolution to \eqref{eq-soliton} if
\[
\frac{v''(t)}{1+v'(t)^2} +\frac{(n-1)f_a'(t)v'(t)}{f_a(t)}-c\ge 0
\]
for $t\ge 0$, where we have denoted $v(x)=: v\big(r(x)\big)$. By a direct computation we get
\[
\frac{v''}{1+v'^2}+\frac{(n-1)f_a' v'}{f_a}-c=\frac{c_1}{n-1}\cdot\frac{\big(1-a^2(f_a/f_a')^2\big)}{1+\left(\tfrac{c_1}{n-1}\right)^2
\left(\tfrac{f_a}{f_a'}\right)^2} +c_1 -c.
\]
By Lemma~\ref{lem2.6} 
\[
1-a(t)^2\big(f_a(t)/f_a'(t)\big)^2 =\big(f_a(t)/f_a'(t)\big)^\prime\to 0
\]
as $t\to\infty$. On the other hand,  $a(t)^2\big(f_a(t)/f_a'(t)\big)^2\to 0$ as $t\to 0+$, hence we obtain that 
\[
\frac{c_1}{n-1}\cdot\frac{\big(1-a^2(f_a/f_a')^2\big)}{1+\left(\tfrac{c_1}{n-1}\right)^2
\left(\tfrac{f_a}{f_a'}\right)^2} +c_1 -c\ge 0
\]
if $c_1$ is large enough. On the other hand, any constant function is a supersolution. Thus the existence of an entire bounded solution follows as in Theorem~\ref{thm-globexist}.
\end{proof}

\subsection{Asymptotic Dirichlet problem}\label{subsec-adp}
Under certain conditions on the functions $a$ and $b$ in the curvature bounds \eqref{curv-as} it is even possible to prescribe the asymptotic behavior of an entire  \emph{bounded} solution. The soliton equation \eqref{eq-soliton} is a special case of the so-called $f$-minimal graph equation
\begin{equation} \label{fmingrapheq}
  \dv \dfrac{\nabla u}{\sqrt{1+|\nabla u|^2}} = \ang{\nb f,\nu},
\end{equation}
where $\bar{\n}f$ is the gradient of a smooth function 
$f\colon N\times\R\to\R$ with respect to the product Riemannian metric of $N\times\R$ and $\nu$ denotes the downward unit normal to the graph of $u$, i.e.
\[
     \nu = \frac{\nabla u-\partial_t}{\sqrt{1+|\nabla u|^2}},
 \]
 where $\partial_t$ denotes the standard coordinate vector field on $R$.
 Indeed, we obtain \eqref{eq-soliton} in the case $f(x,t)=-ct$.
 
 Asymptotic Dirichlet problem for \eqref{fmingrapheq} was solved in \cite{CHH2} under assumptions on $f$ and the radial curvature functions $a$ and $b$ that are not directly applicable in the setting of the current paper. In \cite{CHH2} the authors applied the assumptions and results from \cite{HoVa} to construct local barriers at the sphere at infinity $\partial_\infty N$. These barriers consist of an angular part and of a radially decaying part. The assumptions on $a$ and $b$ are needed to control effectively first and second order derivatives of the barriers. We refer to \cite{CHH2} and \cite{HoVa}  for the definition of the sphere at infinity and other relevant notions concerning the asymptotic Dirichlet problem.
 
Since we are looking for, first of all, bounded solutions, it is natural to assume that $f_a/f_a^\prime$ is integrable; see \eqref{comp-behav} and Theorem~\ref{thm-bounded}. In the light of \cite[Lemma 4.3]{CHH2} we assume that  
\begin{equation}\label{a1}
\lim_{t\to\infty}\frac{f_a(t)t^{1+\varepsilon}}{f_a^\prime(t)}=0
\end{equation}
 for some $\varepsilon>0$.
Scrutinizing the reasoning in \cite{HoVa} we see that 
the assumptions on $a$ and $b$ can be weakened to the setting of the current paper. We assume that the radial curvature functions $a$ and $b\ (\ge a)$ are increasing,
\begin{equation}\label{b1}
\lim_{t\to\infty}\frac{b^\prime(t)}{b(t)^2}=0
\end{equation}
 and that, for each $k>0$, there exist positive and finite limits
 \begin{equation}\label{b2}
 \lim_{t\to\infty}\frac{b\left(t\pm\tfrac{k}{b(t)}\right)}{b(t)}=: c_{\pm k}
 \end{equation}
 and
 \begin{equation}\label{ab1}
  \lim_{t\to\infty}\frac{f_a\left(t-\tfrac{k}{b(t)}\right)}{f_a(t)}>0.
 \end{equation}
 Furthermore, we assume that there exists a constant $\kappa>0$ such that
 \begin{equation}\label{ab2}
 \lim_{t\to\infty}\frac{t^{1+\kappa}b(t)}{f_a^\prime (t)}=0. 
 \end{equation}
The assumption \eqref{b2} will be used instead of \cite[Lemma 3.10]{HoVa} that was one of the important tools in \cite{HoVa} to obtain bounds for first and second order derivatives of the barriers. On the other hand, 
\eqref{ab1} replaces the assumption \cite[(A2)]{HoVa} in the proof of another important tool \cite[Lemma 3.15]{HoVa}.  
 Then we can construct barriers in suitable truncated cones as in \cite[Section 4.1]{CHH2} and apply \cite[Lemma 4.3 and 4.7]{CHH2} in order to solve the following asymptotic Dirichlet problem:
 \begin{thm}\label{thm-adp}
 Suppose that the radial curvature functions $a$ and $b$ satisfy the assumptions \eqref{a1}-\eqref{ab2}. Then, for every continuous $\varphi\in C(\partial_\infty N)$, there exists a unique solution $u\in C^2(N)\cap C(\overline{N})$ to the asymptotic Dirichlet problem
 \begin{equation}\label{dprobinN}
 \begin{cases}
  \dv \dfrac{\nabla u}{\sqrt{1+|\nabla u|^2}} = \frac{c}{\sqrt{1+|\nabla u|^2}} \quad \text{in } N \\
  u|\partial_\infty N = \varphi.
 \end{cases}
\end{equation}
 \end{thm}

 For instance, if $f_a(t)=\sinh(\sinh t)$ and 
  \[
  b(t)^2 =2\cosh(\cosh t),
\]
then the assumptions \eqref{a1}-\eqref{ab2} hold. In this case we have radial curvature bounds
 \[
- 2\cosh\big(\cosh r(x)\big)\le K(P_x)\le
- \cosh^2 r(x) -\sinh r(x)\coth\big(\sinh r(x)\big).
 \]
 
\subsection{Global barriers and applications to the construction of bowl solitons}\label{subsec-gencase}
In this section, we are going to construct global barriers for our problem under certain assumptions on the Riemannian metric that are more general than those in Subsection~\ref{subsec-asrot}. To do so, we first focus on the asymptotics at infinity of solutions to the soliton  equation \eqref{eq-soliton}. Our idea is to make (implicit) assumptions on the metric in order to mimic the asymptotic behavior of radial solutions. In particular, we want to neglect all the non-radial terms in the equation.
We will see in the following that we are able to construct suitable metrics with pinched sectional curvature. Thanks to the knowledge of this asymptotic behavior, we are able to construct sub- and supersolutions to our equation at infinity in such a way that the difference between them goes to zero. To extend these barriers to the whole manifold, we only match these roughly using cut-off function with radial sub- and supersolutions on compact sets. 

We start with noticing that the Riemannian metric on $N$ can be written as
\[
ds^2=dr^2 + d\vartheta^2,
\]
where $r$ is the distance to a fixed point $o\in N$ and $d\vartheta^2$ is the (induced) Riemannian metric on the geodesic sphere $S(r):=S(o,r)$. Then 
\[
\Delta u=u_{rr}+u_{r}\Delta r +\Delta^{S(r)}u,
\]
where $u_{rr}=\partial_r(\partial_r u)$ and $\Delta^{S(r)}$ is the Laplacian on the Riemannian  submanifold $S(r)$. It follows that
the soliton equation \eqref{eq-soliton} is equivalent to
\begin{equation}\label{equiv-eq-sol}
u_{rr}+u_{r}\Delta r +\Delta^{S(r)}u-\frac{\Hess u(\nabla u,\nabla u)}{1+|\nabla u|^2}=c.
\end{equation}  

We write 
$$M(u)=u_{rr}+ u_r \Delta r + \Delta^{S(r)} u - H(u)-c,$$
where 
\[
H(u)=\dfrac{\langle \nabla u,\nabla W^2 \rangle}{2W^2} 
=\frac{\Hess u(\nabla u,\nabla u)}{1+|\nabla u|^2},\ 
W=\sqrt{1+|\nabla u|^2}.
\] 
We also set
$$E(u)=u_{rr}+ \Delta^{S(r)} u - H(u) ,$$
and
$$\tilde{E}(u)= u_{rr}+ \Delta^{S(r)} u .$$

For the next proposition we define a sequence of functions $v_i\colon N\to\R$ inductively by setting 
\[
v_0(x)=\int_{0}^{r(x)}\frac{c\,dt}{\Delta r\big(\gamma(t)\big)},\quad x\in N,
\]
and 
\[
v_i(x)=\begin{cases}-\int_{0}^{r(x)}\frac{M\left(\sum_{j=0}^{i-1}v_{j}\big(\gamma (t)\big)\right)}{\Delta r\big(\gamma(t)\big)}dt,& \text{if }\ \frac{M\left(\sum_{j=0}^{i-1}v_{j}\big(\gamma (t)\big)\right)}{\Delta r\big(\gamma(t)\big)}\notin L^1 (\R ) \\ \int_{r(x)}^{\infty}\frac{M\left(\sum_{j=0}^{i-1}v_{j}\big(\gamma (t)\big)\right)}{\Delta r\big(\gamma(t)\big)}dt,&\ \text{otherwise} \end{cases},\quad i\ge 1,
\]
where $\gamma$ is the unique unit speed geodesic joining $o=\gamma(0)$ and $x=\gamma\big(r(x)\big)$.

In what follows the notation $f=o_R (g)$ means that $\lim_{R\to\infty}f/g=0$.
\begin{prop}\label{prop-main}
Suppose that
\begin{enumerate}
\item[(i)]
\[
E ((1\pm \varepsilon) v_i)= o\left(M\left(\sum_{j=0}^{i-1} v_j \right)\right)
\]
and
\[
H \left(\sum_{j=0}^{i-1} v_j +(1\pm \varepsilon) v_i\right) - H\left(\sum_{j=0}^{i-1} v_j \right)- H\left((1\pm \varepsilon) v_i \right)= 
o\left(M \left(\sum_{j=0}^{i-1} v_j \right) \right)
\]
for some $\varepsilon >0$, and
\item[(ii)] there exists $i_0\in\N$ such that
\[
v_{i_0}(x)\to 0\ \text{ as }\ r(x)\to\infty,
\]
and that
\item[(iii)] for some very large $R_1$, $E(v_{i_0})$ has a constant sign on $N\setminus B(o,R_1)$.
\end{enumerate}
Then there exists $R>0$ such that
\[
M(u_1) \leq 0 \leq M(u_2)\ \text{ on }\ N\setminus B(o,R),
\]
where 
\[
u_k =\sum_{j=0}^{i_0-1}v_j +\big(1+(-1)^{k-1}\varepsilon\big)v_{i_0} =:
\sum_{j=0}^{i_0-1} v_j + (1\pm \varepsilon) v_{i_0},\quad k=1,2.
\]
Moreover, letting
$$g=u_1 -u_2,$$
we have that $|g(x)|\rightarrow 0$ as $r(x)\rightarrow \infty$. Furthermore, if for some $\varepsilon>0$, there exist  radial functions $F_{-}, F_{+}$ such that
$$M(F_{-})\ge\varepsilon,\ M(F_{+})\le -\varepsilon,\ 
\text{ and } F_{-}^{\prime \prime},F_{+}^{\prime \prime}
= o_R (1)\ \text{ for }\ r(x)\ge R.
$$ 
Then, there exist $U_1$ and $U_2$ such that
$$M(U_1)\leq 0\leq M(U_2)\ \text{ on }\ M,$$
and 
$$U_1(x)-U_2 (x)\rightarrow 0\ \text{ as }\ r(x)\rightarrow \infty .$$
\end{prop}

\begin{proof}
We proceed by induction. 
Notice that, for any functions $v$ and $w$, we have
\begin{align*}
M(v+w)&= \Delta v +\Delta w - H(v+w)-c \\
&= M(v) +H(v) +\Delta w - H(v+w)\\
&=M(v)+E (w) + w_r \Delta r - \big(H(v+w) - H(v)-H(w)\big).
\end{align*}
Next we substitute $v$ by $\sum_{j=0}^{i-1} v_j $ and $w$ by $(1\pm \varepsilon ) v_i $, for $\varepsilon >0$. We get
\begin{align*}
M &\left(\sum_{j=0}^{i-1} v_j +(1\pm \varepsilon)v_i\right)  = E( (1 \pm \varepsilon)v_i) \pm \varepsilon M \left(\sum_{j=0}^{i-1} v_j (s)\right)\\
& -\left( H \left(\sum_{j=0}^{i-1} v_j +(1\pm \varepsilon) v_i\right) - H\left(\sum_{j=0}^{i-1} v_j \right)-H ((1\pm \varepsilon) v_i)\right).
\end{align*}
 Since by assumption, 
 \begin{align*}
 E ((1\pm \varepsilon) v_i)
 &= o\left(M \left(\sum_{j=0}^{i-1} v_j \right)\right)\\
 &= H\left(\sum_{j=0}^{i-1} v_j +(1\pm \varepsilon)v_i\right) 
 - H\left(\sum_{j=0}^{i-1} v_j \right) -H ((1\pm \varepsilon) v_i), 
 \end{align*} 
 one can show that $\sum_{j=0}^{i-1} v_j + (1\pm \varepsilon) v_i $ is a sub- or supersolution depending on the sign of $M \big(\sum_{j=0}^{i-1} v_j \big) $. We stop as soon as $v_i (x) \rightarrow 0$ as $r(x)\rightarrow \infty$.

 Next we construct a global sub- and supersolutions. To do so, let $\chi(x)$ be a radial function such that 
\[
 \chi (x)=\begin{cases}1,& \text{if}\ r(x)\leq A;\\ 0,& \text{if}\ r(x)\geq B \end{cases}
 \] 
 for some $A,B$ to be determined later. 
 We set $U_2= \chi F_{-} + (1-\chi) u_2$. If $B\geq R$, then it is easy to see that $M(U_2)\geq 0 $ on $B(o,A) \cup \big(N\setminus B(o,B)\big)$. By assumption, we have that $\big(\tilde{E} (u_2)-H(u_2)\big)(x)\rightarrow 0 $ and $F_{-}^{\prime \prime}(x)\rightarrow 0$ as $r(x)\rightarrow \infty$. So, for $r(x)\geq A$, we have
$$M (U_2) = \frac{(U_2)_{rr}}{1+ |\nabla U_2|^2}+(U_2)_r \Delta r -c +o_A (1).$$
Next, we choose $B$ such that 
\begin{align*}
(U_2)_r & = \chi^\prime (F_{-} - u_2) + \chi F_{-}^\prime +(1-\chi) (u_2)_r \\
& = \chi F_{-}^\prime +(1-\chi) (u_2)_r +o_R (1/\Delta r )
\end{align*}
and 
\begin{align*}
(U_2)_{rr} &= \chi^{\prime \prime} (F_{-} - u_2) + 2\chi' (F_{-}^\prime -(u_2)_r) + \chi F_{-}^{\prime\prime} +(1-\chi) (u_2)_{rr} \\
& =\chi F_{-}^{\prime\prime} +(1-\chi) (u_2)_{rr} +o_R (1/ \Delta r).
\end{align*} 
 Thanks to this choice, we get
 \begin{align*}
 M(U_2)=\ &\frac{ \chi F_{-}^{\prime\prime} +(1-\chi) (u_2)_{rr} +o_R (1/\Delta r)}{1+ |\nabla U_2|^2}\\
 & +  
\big(\chi F_{-}^\prime +(1-\chi)(u_2)_r\big)\Delta r -c +o_R (1). 
 \end{align*} 
 By assumptions, we have
$$\frac{ \chi F_{-}^{\prime\prime} +(1-\chi)(u_2)_{rr}+o_R (1/\Delta r)}{1+ |\nabla U_2|^2}=o_R (1),$$
and
$$(u_2)_r \Delta r -c=o_R(1).$$
Since $M(F_{-})= F_{-}^\prime \Delta r -c +o_R (1) \geq \varepsilon$, we conclude that
$$M(U_2)= \chi (F_{-}^\prime -(u_2)_r) \Delta r +o_R(1)\geq \varepsilon/2 . $$
We proceed in the same way to prove that $U_1 = \chi F_{+} + (1-\chi) u_1$ is a global supersolution. 
\end{proof}

\begin{thm}\label{thm-main2} Let $(N,g)$ be a Cartan-Hadamard manifold whose Riemannian metric $g$ admits functions $v_i,\ i\ge 0,\ F_{-}$, and $F_{+}$ satisfying the assumptions in Proposition~\ref{prop-main}. Then there exists a global bowl soliton on $(N,g)$.
\end{thm}

\begin{proof}
Let $U_1$ be an entire supersolution and $U_2$ an entire subsolution, with $U_1(x)-U_2(x)\to 0$ as 
$r(x)\to\infty$,
provided by Proposition~\ref{prop-main}. It follows from the comparison principle that $U_1(x)\ge U_2(x)$ for every $x\in N$. Indeed, if $U_1(y_0)<U_2(y_0)$ for some 
$y_0\in N$, let $D$ be the $y_0$-component of the set 
$\{x\in N\colon U_1(x)<U_2(x)-\delta/2\}$, with $\delta=U_2(y_0)-U_1(y_0)>0$. Then $D$ is an open relatively compact subset of $N$ and $U_1(x)=U_2(x)-\delta/2$ on $\partial D$. Hence $U_1(x)=U_2(x)-\delta/2$ in $D$ by the comparison principle leading to a contradiction. Now the sets 
\[
\Omega_j=\{x\in N\colon U_1(x)>U_2(x)+1/j\},
\]
for all sufficiently large $j$, exhaust $N$. We can use $U_1$ and $U_2+1/j$ as upper and lower barriers in $\Omega_j$ in a similar fashion than $u_b$ and $u_a$ in Theorem~\ref{thm-globexist} to obtain a global solution. 
\end{proof}

Next, we use the previous results to prove the existence of bowl solitons on two manifolds with pinched sectional curvatures. The sectional curvatures of the first one are pinched between two arbitrary negative constants (which are attained for some planes) while the sectional curvatures of the second one go quadratically to $0$ at infinity and take the values $\alpha^2/r^2$ and $\beta^2/r^2$, $\beta>\alpha>4\sqrt{5}$, for some planes. 

\begin{cor}\label{firstcor}
Let $(N^2,g)$ be a 2-dimensional Cartan-Hadamard manifold  with the Riemannian metric $g =dr^2 + h(r,\theta)^2 d\theta^2$, where 
\[
h(r,\theta)=\frac{1}{a\cos^2 \theta +b\sin^2 \theta}\sinh (ar\cos^2 \theta +b r \sin^2 \theta).
\] 
Then the assumptions of Proposition~\ref{prop-main} are satisfied. In particular, there exists a bowl soliton on $(N^2,g)$.
\end{cor}

\begin{proof}
In this metric, noticing that $\Delta r= h_r/h$, our equation rewrites as
\begin{align*}
&M(u) +c= u_{rr} +u_r \frac{h_r}{h}+ \frac{u_{\theta \theta}}{h^2} -\\ 
&\frac{(u_r)^2 u_{rr} +h^{-2} u_r u_\theta u_{r\theta} - u_\theta^2 u_r h_r h^{-3} + h^{-2}u_r u_{r\theta} u_\theta + u_\theta^2 u_{\theta  \theta } h^{-4} - u_\theta^3 h_\theta h^{-3}}{1+(u_r)^2 + h^{-2} u_\theta^2} .
\end{align*}
In this metric, $v_0$ is given by
$$v_0 (x)=c \int_0^r \dfrac{ h(\gamma (t))}{h_r (\gamma (t))}dt.$$
From this expression, we observe that $v_0(x)\leq C r(x)$,
$$(v_0)_r (x) = c \dfrac{h(x)}{h_r (x)} = \dfrac{c\, \tanh (ar\cos^2 \theta +br \sin^2 \theta)}{a \cos^2 \theta +b \sin^2 \theta}  \leq C,$$
 and 
\[
(v_0)_{rr}(x)=c/h_r^2(x)= \dfrac{c}{\cosh^2 (ar \cos^2 \theta +b r\sin^2 \theta ) } \leq c (\max \{e^{-ar},e^{-br} \} )^{\tilde{c}},
\]
for some $\tilde{c}>0$. Concerning the spherical derivatives, very rough estimates show that
$$|(v_0)_{\theta} (x)|,\ |(v_0)_{\theta \theta} (x)|,\ |(v_0)_{r \theta} (x)|\leq Cr(x).$$
Using the previous estimates and the fact that $h$ decays exponentially, we deduce that there exists a constant $\tilde{c}$ such that
\begin{align*}
& \left| u_{rr} + \frac{u_{\theta \theta}}{h^2} -\frac{(u_r)^2 u_{rr} +h^{-2} u_r u_\theta u_{r\theta}+ h^{-2}u_r u_{r\theta} u_\theta + u_\theta^2 u_{\theta  \theta } h^{-4}}{1+(u_r)^2 + h^{-2} u_\theta^2}\right|\\
& \leq   (\max \{e^{-ar},e^{-br} \} )^{\tilde{c}}.
\end{align*}
Noticing that $|h_r h^{-3}|, \ |h_\theta h^{-3}| \leq h^{-2}$, we also have
$$| u_\theta^2 u_r h_r h^{-3}|,\ | u_\theta^3 h_\theta h^{-3}|  \leq   (\max \{e^{-ar},e^{-br} \} )^{\tilde{c}}.$$
Finally, since $u_r h_r /h =c$, we obtain that
$$|M(v_0)| \leq  (\max \{e^{-ar},e^{-br} \} )^{\tilde{c}}. $$
This proves that (i) of Proposition \ref{prop-main} holds for $i=0$. Thanks to this estimate, we get that
\begin{align*}
|v_1 (x)| &= \left|\int_{r(x)}^\infty M\big(v_0 (\gamma (t)\big) \dfrac{h (\gamma (t))}{h_r (\gamma (t))} dt\right|
\\
&\leq C \int_{r(x)}^\infty  \big(\max \{e^{-a\gamma (t)},e^{-b\gamma (t)} \} \big)^{\tilde{c}} dt. 
\end{align*}
Therefore, $v_1 (x)\rightarrow 0$ as $r(x)\rightarrow \infty$ which is (ii) of Proposition \ref{prop-main}.  Finally, notice that $v_1$ and none of the terms of $E(v_1)$ oscillates at infinity so clearly $E(v_1(x))$ has a sign provided that $r(x)$ is large enough. Thus (iii) of Proposition \ref{prop-main} holds. To conclude the proof, we are left with the existence of $F_{\pm}$. Recall from de Lira and Martin \cite{LM}, that if $P$ is a rotationally symmetric manifold with constant sectional curvature, there exists a radially symmetric function $w_P$ such that $M(w_P)=0$, for any $c_P\in \R$ and $(w_P)_{rr}(x)$ goes to $0$ when $r(x)$ goes to infinity. Since the sectional curvatures of our manifold are pinched between $-a^2$ and $-b^2$, $F_{\pm}$ can be obtained by taking $w_P$ for some appropriate manifolds $P$ and some constant $c_P$ by standard comparison theorems. This proves the corollary.
\end{proof}

\begin{cor}\label{secondcor}
Let $(N^2,g)$ be a 2-dimensional Cartan-Hadamard manifold  with the Riemannian metric $g =dr^2 + h(r,\theta)^2 d\theta^2$, where 
\[
h(r,\theta)=\cos^2 \theta f_a (r) + \sin^2 \theta f_b (r),
\]
 where $f_b''(r)=\frac{\beta^2}{r^2}f_b(r)$, 
  $f_a'' (r)= \frac{\alpha^2}{r^2} f_a (r)$, with  $\beta>\alpha>4\sqrt{5}$. Then the assumptions of the previous proposition are satisfied. In particular, there exists a bowl soliton on $(N^2,g)$.
\end{cor}

\begin{proof}
To simplify notation, we take $c=1$. In this case, it is easy to see that $|v_0 (x)|,\ |(v_0)_\theta (x)|,\ |(v_0)_{\theta \theta} (x)|\leq C r(x)^2 $, $|(v_0)_r (x)|, |(v_0)_{r\theta} (x)|\leq C r(x)$ and $|(v_0)_{rr}(x)|\leq C$. We have
\begin{align*}
M(v_0) &= \dfrac{(v_0)_{rr}}{1+(v_0)_r^2 + h^{-2} (v_0)_\theta^2} + \frac{(v_0)_{\theta \theta}}{h^2} \\ 
&-\frac{h^{-2} (v_0)_r (v_0)_\theta (v_0)_{r\theta} - (v_0)_\theta^2 (v_0)_r h_r h^{-3} + h^{-2}(v_0)_r (v_0)_{r\theta} (v_0)_\theta}{1+((v_0)_r)^2 + h^{-2} (v_0)_\theta^2}\\
& + \frac{(v_0)_\theta^2 (v_0)_{\theta  \theta } h^{-4} - (v_0)_\theta^3 h_\theta h^{-3}}{1+((v_0)_r)^2 + h^{-2} (v_0)_\theta^2}.
\end{align*}
Since 
\[
f_a (r)\approx r^{\frac{1+\sqrt{1+\alpha^2}}{2}}
\]  when $r$ is large enough and $\alpha ,\beta >4\sqrt{5}$, we get that 
\begin{equation}
\label{hsmalle1}
\dfrac{r^5}{h}=o(1).
\end{equation}
 This implies in particular that 
$$ h^{-2} (v_0)_\theta^2 =o( (v_0)_r^2 ),$$
 where we used that $(v_0)_r (x)\approx r (x)$ when $r$ is large enough. Next, we claim that
$$M(v_0)(x) \approx  \dfrac{(v_0)_{rr}}{1+(v_0)_r^2 + h^{-2} (v_0)_\theta^2} \approx  \dfrac{(v_0)_{rr}}{(v_0)_r^2 } \approx r^{-2}(x), $$
where we used that $(v_0)_{rr} (x) \approx 1 $. Indeed, thanks to \eqref{hsmalle1}, we see that
$$ \frac{(v_0)_{\theta \theta}}{h^2}=o( r^{-2} ). $$ 
Since
\begin{gather*}
|(v_0)_r (v_0)_\theta (v_0)_{r\theta}|,\ |(v_0)_r (v_0)_{r\theta} (v_0)_\theta |\leq r^4,\\
|(v_0)_\theta^2 (v_0)_r  |\leq r^5,\\
|(v_0)_\theta^2 (v_0)_{\theta  \theta } |,\ | (v_0)_\theta^3 h_\theta h^{-1}| \leq r^6,
\end{gather*}
and $|h_r /h| \leq 1/r$ and $|h_\theta /h|\leq C$, using once more \eqref{hsmalle1}, we deduce that 
\begin{align*}
B(v_0)&= h^{-2} (v_0)_r (v_0)_\theta (v_0)_{r\theta} - (v_0)_\theta^2 (v_0)_r h_r h^{-3} \\
&\quad + h^{-2}(v_0)_r (v_0)_{r\theta} (v_0)_\theta + (v_0)_\theta^2 (v_0)_{\theta  \theta } h^{-4} - (v_0)_\theta^3 h_\theta h^{-3}\\
& =o(1).
\end{align*}
This proves the claim that
$$M(v_0)(x) \approx \dfrac{(v_0)_{rr}}{(v_0)_r^2 } \approx r^{-2}(x).$$
So, we can write $v_1$ as
$$v_1 (x)=- \int_0^r \left(\dfrac{M(v_0) h}{h_r}\right) \big(\gamma (t)\big) dt \approx  -\int_0^r \left( (\dfrac{h}{h_r})_r \dfrac{ h_r}{h}\right) \big(\gamma (t)\big) dt. $$
So we see that 
\begin{gather*}
|v_1 (x)|,\ |(v_1)_\theta (x)|,\ |(v_1)_{\theta \theta} (x)|\leq C \log r(x),\\
 |(v_1)_r (x)|, |(v_1)_{r\theta} (x)|\leq C
 \end{gather*} 
 and $|(v_1)_{rr}(x)|\leq C r^{-1}(x)$. Now, a direct computation gives that (compare with \eqref{limeta})
\begin{align*}
M(v_0 +v_1) = &\Bigl[\frac{3(n-1)}{c^3}\left(\left(\frac{h}{h'}\right)'\right)^2 - \frac{(n-1)^2}{c^3}\left(\frac{h}{h'}\right)'\\
&\ -\frac{n-1}{c^3}\left(\frac{h}{h'}\right)\left(\frac{h}{h'}\right)'' \Bigr]  \left(\frac{h'}{h}\right)^4 \\
&\ +o(r^{-4}) +\Delta^{S(r)} (v_0 +v_1) - \dfrac{B(v_0 +v_1)}{1+ |\nabla (v_0 +v_1)|^2}. 
\end{align*}
Using \eqref{hsmalle1}, one can check as previously that 
$$|\Delta^{S(r)} (v_0 +v_1)|+ |\dfrac{B(v_0 +v_1)}{1+ |\nabla (v_0 +v_1)|^2}| =o(r^{-4}).$$
We deduce from this that
\begin{align*}
v_2 (x)\approx &  \int_r^\infty \Bigl[\frac{3(n-1)}{c^3}\left(\left(\frac{h}{h'}\right)'\right)^2 \\
& - \frac{(n-1)^2}{c^3}\left(\frac{h}{h'}\right)'-\frac{n-1}{c^3}\left(\frac{h}{h'}\right)\left(\frac{h}{h'}\right)'' \Bigr]  \left(\frac{h'}{h}\right)^3  dt. 
\end{align*}
By construction, we have that (i) of Proposition \ref{prop-main} holds true.
Observe that $v_2 (x) \rightarrow 0$ as $r(x)\to \infty$ therefore (ii) of Proposition \ref{prop-main} holds true for $\tilde{i}=2$. As in the previous corollary, none of the terms of $E(v_2)$ oscillate so it has a sign at infinity i.e. (iii) holds. To prove the existence of $F_\pm$, we proceed as in the previous corollary. Let us point out that  the condition $F^{\prime \prime}_{\pm}\rightarrow 0$ of Proposition \ref{prop-main} can be replaced by $F^{\prime \prime}_{\pm}/\big(1+|\nabla U_2|^2\big)\rightarrow 0$ which is satisfied in our current situation since $|\nabla U_2|\rightarrow \infty$.
\end{proof}


\end{document}